\documentclass [11pt] {article}
\usepackage{amsfonts}
\usepackage{color}
\usepackage{amsthm}
\usepackage{graphicx}
\usepackage{url}
\usepackage{txfonts}
\usepackage{algorithm,algorithmic}
\usepackage{epstopdf}
\usepackage{booktabs}
\usepackage[justification=centering]{caption}

\renewcommand{\baselinestretch}{1.2}
\newtheorem{theorem}{Theorem}[section]
\newtheorem{proposition}[theorem]{Proposition}
\newtheorem{lemma}[theorem]{Lemma}

\newtheorem{definition}[theorem]{Definition}

\renewcommand{\thefootnote}

\newcommand{\be}{\begin{equation}}
\newcommand{\ee}{\end{equation}}

\begin{document}

\renewcommand{\baselinestretch}{1.2}

\title {Feasible Newton's methods for symmetric tensor $Z$-eigenvalue problems}

\author{Dong-Hui Li\thanks{Email: lidonghui@m.scnu.edu.cn.} \ \ \  Xueli Bai\thanks{Email: 20200592@m.scnu.edu.cn} \ \ and \ \ Jiefeng Xu\thanks{Email: 2018021699@m.scnu.edu.cn}\\
School of Mathematical Sciences, \\
South China Normal University, Guangzhou, 510631, China }

\maketitle

\begin{abstract}
  Finding a $Z$-eigenpair of a symmetric tensor is equivalent to finding a KKT point of a sphere constrained minimization problem. Based on this equivalency, in this paper, we first propose a class of iterative methods to get a Z-eigenpair of a symmetric tensor. Each method can generate a sequence of  feasible points such that the sequence of function evaluations is decreasing. These methods can be regarded as extensions of the descent methods for unconstrained optimization problems. We pay particular attention to the Newton method. We show that under appropriate conditions, the Newton method is globally and quadratically convergent. Moreover, after finitely many iterations, the unit steplength will always be accepted. We also propose a nonlinear equations based Newton's method and establish its global and quadratic convergence. In the end, we do several numerical experiments to test the proposed Newton methods. The results show that both Newton methods are very efficient.
\end{abstract}


\section{Introduction}\label{intro}

Let $\mathbb{R}$ be the real field. ${\cal A}$ is called an $m$-th order $n$-dimensional real tensor if it takes the form
$$
{\cal A}=(a_{i_1i_2\ldots i_m}),\quad a_{i_1i_2\ldots i_m}\in\mathbb{R},\quad 1\leq i_1,i_2\cdots,i_m\leq n.
$$
We use $\mathbb{R}^{[m,n]}$ to denote the set of all real tensors of order $m$ and dimension $n$. If the elements of ${\cal A}$ are invariant under arbitrary permutation of their indices, then ${\cal A}$ is called a symmetric tensor. The set of all $m$-th order $n$-dimensional symmetric tensors is denoted by $\mathbb{S}^{[m,n]}$. For a tensor ${\cal A}=(a_{i_1i_2\ldots i_m})\in\mathbb{R}^{[m,n]}$ and a vector
$x=(x_1,\ldots,x_n)^T\in\mathbb{R}^n$, ${\cal A}x^{m-1}$ denotes a $n$-dimensional vector whose $i$-th element is defined as
$$
({\cal A}x^{m-1})_i:=\sum^n_{i_2,\ldots,i_m=1}a_{ii_2\ldots i_m}x_{i_2}\cdots x_{i_m},\quad \forall i\in\{1,2,\ldots,n\}.
$$
And the notation ${\cal A}x^{m}$ denotes a homogeneous polynomial of degree $m$:
$$
{\cal A}x^{m}:=\langle x,{\cal A}x^{m-1}\rangle=\sum^n_{i_1,\ldots,i_m=1}a_{i_1\ldots i_m}x_{i_1}\cdots x_{i_m}.
$$
It is well-known that if $\cal A$ is symmetric, then, the gradient of the homogenous function ${\cal A}x^m$ is $m{\cal A}x^{m-1}$, i.e.,
\[
\nabla \Big ({\cal A}x^m \Big )=m{\cal A} x^{m-1},\quad\forall x\in \mathbb R^n,
\]
and the Hessian is $m(m-1){\cal A}x^{m-2}$, where  ${\cal A}x^{m-2}\in\mathbb{R}^{n\times n}$ whose $(i,j)$-th element is given by
$$
({\cal A}x^{m-2})_{ij}:=\sum^n_{i_3,\ldots,i_m=1}a_{iji_3\ldots i_m}x_{i_3}\cdots x_{i_m},\quad \forall i,j\in\{1,2,\ldots,n\}.
$$

The concept of tensor eigenvalues was introduced by Qi and Lim, independently.
\begin{definition}\label{eig}{\rm (\cite{Lim05,Qi05})}
Let ${\cal A}\in\mathbb{R}^{[m,n]}$. We say that $\lambda\in\mathbb{R}$ is an $H$-eigenvalue
and $x\in\mathbb{R}^n\backslash\{0\}$ is the corresponding $H$-eigenvector of ${\cal A}$ if
\begin{eqnarray*}\label{Heig}
{\cal A}x^{m-1}=\lambda x^{[m-1]},
\end{eqnarray*}
where $x^{[m-1]}:=(x_1^{m-1},\ldots,x_n^{m-1})^T\in\mathbb{R}^n$. $\lambda$ is called a
$Z$-eigenvalue and $x$ is the corresponding $Z$-eigenvector of ${\cal A}$ if
\begin{equation}\label{Zeig}
{\cal A}x^{m-1}=\lambda x\quad\mbox{and}\quad x^Tx=1.
\end{equation}
\end{definition}

In this paper, we focus on tensor $Z$-eigenvalues, which is closely related to many practical applications,
such as blind source separation \cite{KR02}, molecular conformation \cite{D88}, magnetic resonance imaging \cite{QWW08}, Bose-Einstein Condensation \cite{YHL19}, and quantum information \cite{NQB14}. Over the past decades, the tensor $Z$-eigenvalue has been studied extensively both in theory and algorithms (see e.g., references \cite{DW15,QL17,QCC18}).

Although, it has been mentioned in \cite{HL13} that, in general, computing eigenpairs of higher order tensors is NP-hard, in recent years, a mass of papers have been contributed to designing effective
algorithms to solve tensor $Z$-eigenvalue problems with different structured tensors involved. Some of them are aimed at finding all $Z$-eigenvalues, for examples, Qi et al. \cite{QWW09} proposed a direct orthogonal transformation $Z$-eigenvalue method for supersymmetric tensors ${\cal A}\in\mathbb{R}^{[3,3]}$; Cui et al. \cite{CDN14} addressed a sophisticated Jacobian semidefinite relaxation method for symmetric tensors; Jaffe et at. \cite{JWN18} gave a framework of a Newton correction iteration method for symmetric tensors. Some of them are with the purpose of finding one $Z$-eigenpair. Kolda and Mayo \cite{KM11} presented a shifted power method for computing real symmetric tensors' $Z$-eigenpairs with linear convergence rate, and then, extended it to the generalized tensor eigenvalue problems in \cite{KM14}; Yu et al. \cite{YYXSZ16} proposed a linearly convergent adaptive gradient method; Guo et al. \cite{GLL18} developed a modified local Newton iteration for computing nonnegative $Z$-eigenpairs of nonnegative tensors, which enjoys locally quadratic convergence property under proper assumptions; Zhang et al. \cite{ZNG20} introduced a Newton method for newly defined almost nonnegative irreducible tensors. The other works are interested in finding the extreme $Z$-eigenvalues. For example, Hu et al. \cite{HHQ13} employed a sequential semidefinite programming method by formulating optimization problem into a tensor conic linear programming; Hao et al. \cite{HCD15} proposed a sequential subspace projection method for symmetric tensors by solving the original optimization problem in a 2-dimensional subspace at the current point, and they also showed a trust region method which enjoys a locally quadratic convergence rate in \cite{HCD152}.

So far, the study in Newton's method is relatively few, especially feasible Newton method. In this paper,
we are interested in designing feasible Newton methods for the $Z$-eigenvalue problem, which is to find a solution $(x,\lambda)$ of the system (\ref{Zeig}). Two kinds of approaches will be considered for this problem. One of them is the optimization based approach, which is to find a Karush-Kuhn-Tucker (KKT) point of the following sphere constrained optimization problem:
$$
\max _{x\in \mathbb{B}} /\min _{x\in \mathbb{B}} \phi (x)= \frac 1m {\cal A}x^m,
$$
where $\mathbb{B}=\{x\in \mathbb R^n\mid x^Tx=1\}$ is the unit sphere centered at the origin. The meaning of $\max/\min$ is that we can choose either to maximize or to minimize $\phi (x)$. The key point is that if ${\cal A}\in\mathbb{S}^{[m,n]}$, then, the $Z$-eigenpair $(x,\lambda)$ of ${\cal A}$ is the same as the KKT point of the above problem. Another approach is the nonlinear equations based method. Define $F(x)={\cal A}x^{m-1}- {\cal A}x^m\cdot x$. Since a $Z$-eigenvector $x$ of ${\cal A}$ satisfies $x^Tx=1$, it is easy to see that the corresponding $Z$-eigenvalue $\lambda$ is $\lambda={\cal A}x^{m}$. Therefore, solving tensor $Z$-eigenvalue problem (\ref{Zeig}) is equivalent to solving the following constrained equations:
$$
F(x)=0,\quad x\in \mathbb{B}.
$$
We will develop two kinds of feasible Newton methods to get a Z-eigenpair of a symmetric tensor from the above two aspects, respectively.

The rest of this paper is organized as follows. In Section \ref{OptNewton}, we first propose a general feasible descent method based on the equivalent optimization reformulation of the $Z$-eigenvalue problem.
Then, we pay particular attention to the Newton method for the same problem. We show that under reasonable conditions, the feasible descent Newton method is globally and quadratically convergent. Especially, the unit steplength is essentially accepted. In Section \ref{EqnNewton}, we present a nonlinear equations based Newton method and also establish its global and quadratic convergence. In Section \ref{num}, we do some numerical experiments to indicate the effectiveness of the proposed algorithms. Conclusions are given in Section \ref{con}.

Throughout this paper, we use lowercases $x,y,z,\ldots$ for vectors, capital letters $A,B,C,\ldots$ for matrices and calligraphic letters ${\cal A}, {\cal B}, {\cal C},\ldots$ for tensors. We denote $[n]=\{1,2,\ldots,n\}$ and $\mathbb{B}=\{x\in \mathbb R^n\mid x^Tx=1\}$. In addition, we use $I\in\mathbb{R}^{n\times n}$ and ${\cal I}\in\mathbb{R}^{[m,n]}$ to denote the identity matrix and identity tensor, respectively. Finally, the Jacobian of a continuously differentiable function $F:\mathbb{R}^n\rightarrow\mathbb{R}^n$ at $x\in\mathbb{R}^n$ is denoted by $F'(x)$, and the gradient and the Hessian of a twice continuously differentiable function $H:\mathbb{R}^n\rightarrow\mathbb{R}$ at $x\in\mathbb{R}^n$ is given by $\nabla H(x)$ and $\nabla^2 H(x)$, respectively.

\section{An optimization based Newton's method for symmetric tensors}\label{OptNewton}

In this section, we first propose a class of iterative methods for finding a $Z$-eigenpair of a symmetric tensor ${\cal A}\in\mathbb{S}^{[m,n]}$, which is equivalent to finding a KKT point of the sphere constrained optimization problem:
\begin{equation}\label{opt}
\min_{x\in \mathbb{B}} \phi (x):=\frac 1m {\cal A}x^m,
\end{equation}
or
\begin{equation}\label{opt-max}
\max _{x\in \mathbb{B}} \phi (x):=\frac 1m {\cal A}x^m,
\end{equation}

As our purpose is to find one $Z$-eigenpair of $\cal A$, we will only focus our attention on problem (\ref{opt}). If the purpose is to find the largest $Z$-eigenvalue and its corresponding eigenvector,
 we can consider to solve problem (\ref{opt-max}) in a similar way.

We first propose a general feasible descent method to find a KKT point of problem (\ref{opt}) in subsection \ref{Opt-descent}. The method generates a sequence of iterates that are feasible in the sense that they are on the sphere $\mathbb B$. And a sufficient condition for the method to be globally convergent is established. In subsection \ref{opt-newton}, we focus on Newton's method and present a feasible Newton method to find a Z-eigenpair of a symmetric tensor. We will show that under reasonable conditions, this method is globally and quadratically convergent. Moreover, the unit steplength will be accepted after finitely many iterations.

\subsection{A general descent method}\label{Opt-descent}

For any $x\in\mathbb{B}$ and $d\in\mathbb{R}^n$, let $x(\alpha)$ be the projection of $x+\alpha d$
onto $\mathbb B$, i.e., $x(\alpha)=\frac {x+\alpha d}{\|x+\alpha d\|}$. Here and throughout the paper,
without specification, we always use $\|x\|$ to denote the Euclid norm of a vector $x$, i.e., $
\|x\|=\sqrt {x^Tx}$. By an elementary deduction, we obtain that, for any $x\in \mathbb B$,
 \begin{eqnarray*}
 &&x(\alpha) = \frac {x+\alpha d}{\|x+\alpha d\|}= (x+\alpha d) + \frac {1-\|x+\alpha d\|}{\|x+\alpha d\|} (x+\alpha d)\\
    &&= (x+\alpha d)+  \frac {1-\|x+\alpha d\|^2}{\|x+\alpha d\| (1+\|x+\alpha d\|)} (x+\alpha d)\\
    &&=  (x+\alpha d) -\frac {2\alpha d^Tx +\alpha ^2\|d \|^2}{\|x+\alpha d\| (1+\|x+\alpha d\|)} (x+\alpha d)\\
    &&= (x+\alpha d) -\frac {2\alpha d^Tx }{\|x+\alpha d\| (1+\|x+\alpha d\|)} x + O(\alpha ^2\|d \|^2)\\
    &&= x+\alpha (I - xx^T)d +\alpha \Big (1 -\frac {2 }{\|x+\alpha d\| (1+\|x+\alpha d\|)}\Big ) (d^Tx) x + O(\alpha ^2\|d \|^2).
\end{eqnarray*}
We estimate the third term of the last equality and get
\begin{eqnarray*}
1 -\frac {2 }{\|x+\alpha d\| (1+\|x+\alpha d\|)} &=&\frac {\|x+\alpha d\| (1+\|x+\alpha d\|)-2 }{\|x+\alpha d\| (1+\|x+\alpha d\|)} \\
    &=& \frac {\|x+\alpha d\|+\|x+\alpha d\|^2-2 }{\|x+\alpha d\| (1+\|x+\alpha d\|)}\\
    &=& \frac {\|x+\alpha d\|-1+2\alpha d^Tx +\alpha ^2\|d\|^2 }{\|x+\alpha d\| (1+\|x+\alpha d\|)}\\
    &=& \frac {3\alpha d^Tx}{\|x+\alpha d\| (1+\|x+\alpha d\|)}+ O(\alpha ^2\|d\|^2).
\end{eqnarray*}
The arguments above give the following relation
\begin{eqnarray}\label{iter2}
x(\alpha)-x=\alpha (I - xx^T)d +O(\alpha ^2\|d\|^2).
\end{eqnarray}
By the use of the mean-value theorem, we can obtain that
\begin{eqnarray}\label{mean}
\phi(x(\alpha))&=&\phi(x)+[\nabla\phi(x)]^T(x(\alpha)-x)+O(\|x(\alpha)-x\|^2)\nonumber \\
&=& \phi(x)+\alpha[\nabla\phi(x)]^T(I-xx^T)d+O(\|\alpha d\|^2).
\end{eqnarray}

For the sake of convenience, we introduce the following concept of curve descent direction, which is an extension of a descent direction.
\begin{definition}\label{def:c-descent}
Let $d\in \mathbb R^n$ and
\[
x(\alpha)=\frac {x+\alpha d}{\|x+\alpha d\|}.
\]
We call $d$ a curve descent direction of function $\phi$ at $x$ if
\[
\phi (x(\alpha))<\phi (x)
\]
holds for all $\alpha>0$ sufficiently small. If $-d$ is a curve descent direction of $\phi $ at $x$,
then, $d$ is called a curve ascent direction of $\phi$ at $x$.
\end{definition}

It is easy to see from (\ref{mean}) that if a direction $d$ satisfies $ [\nabla\phi(x)]^T(I-xx^T)d<0$,
then, it is a curve descent direction of $\phi$ at $x$. In that case, for any constant $\sigma\in(0,1)$, the inequality
\begin{equation}\label{search:cond}
\phi(x(\alpha))\le \phi(x)+\alpha\sigma \, [\nabla\phi(x)] ^T(I-xx^T)d
\end{equation}
holds for all $\alpha>0$ sufficiently small.

Observing that $x(\alpha)\in \mathbb B$ for all $\alpha\in \mathbb R$, the above process motives us to
propose a feasible descent method for finding a $Z$-eigenpair of a symmetric tensor $\cal A$. Specifically, at iteration $k$, suppose that we have already had a feasible point $x_k\in \mathbb B$, we then find a curve descent direction $d_k$ satisfying $[\nabla\phi(x_k)]^T(I-x_kx_k^T)d_k<0$. By using a backtracking line search technique, we determine a steplength $\alpha_k>0$ satisfying the inequality (\ref{search:cond}) with $x=x_k$, $d=d_k$ and $\alpha=\alpha_k$. The next iterate $x_{k+1}$ is then determined by $x_{k+1}=x_k(\alpha_k)$, which is also a feasible point.

The key of the above descent method is the determination of the curve descent direction $d_k$. Similar to
the descent method in the unconstraint optimization, we give a general choice for $d_k$:
\begin{eqnarray}\label{direction-opt}
d_k=-B_k^{-1}(I-x_kx_k^T)\nabla\phi(x_k)=-B_k^{-1}F (x_k),
\end{eqnarray}
where $B_k\in\mathbb{R}^{n\times n}$ is symmetric and positive definite and
$F(x)={\cal A}x^{m-1}-{\cal A}x^m\cdot x=(I-xx^T){\cal A}x^{m-1}$. It is easy to see that such a direction $d_k$ satisfies
\begin{eqnarray*}
[\nabla\phi(x_k)]^T(I-x_kx_k^T)d_k&=& - [\nabla \phi (x_k)]^T (I-x_kx_k^T)B_k(I-x_kx_k^T)\nabla \phi (x_k)\\
&=&- [F(x_k)]^TB_kF(x_k)\le 0.
\end{eqnarray*}
Moreover, $[\nabla\phi(x_k)]^T(I-x_kx_k^T)d_k=0$ if and only if $F(x_k)=0$, which means that $x_k$ is a $Z$-eigenvector of $\cal A$ with corresponding eigenvalue $\lambda _k={\cal A}x^m_k$. Hence, we have the following proposition.

\begin{proposition}\label{dec-dir}
Given a vector $x_k\in \mathbb B$ and a symmetric and positive definite matrix $B_k\in\mathbb{R}^{n\times n}$. Suppose that $x_k$ is not a $Z$-eigenvector of ${\cal A}$. Then, $d_k$ defined by (\ref{direction-opt}) is a curve descent direction of $\phi $ at $x_k$ which satisfies
\[
[\nabla\phi(x_k)]^T(I-x_kx_k^T)d_k< 0.
\]
Consequently, the inequality (\ref{search:cond}) with $x=x_k$ and $d=d_k$ holds for all $\alpha>0$
sufficiently small.
\end{proposition}

The general feasible descent method for finding a $Z$-eigenpair of a symmetric tensor ${\cal A}$ is stated below in Algorithm \ref{alg1}.

\begin{algorithm}[!htbp]
\caption{(A general descent method for problem (\ref{opt})).}\label{alg1}
\begin{algorithmic}[1]
\STATE Given a constant $\sigma\in (0,1)$, an initial point $x_0\in \mathbb{B}$ and a symmetric and positive definite matrix sequence $\{B_k\}$.
Let $k:=0$ and $\lambda_0={\cal A}x_0^m$.
\WHILE{$F(x^k)\neq 0$}
\STATE Let $x_k(\alpha)$ and $d_k$ be defined by
$$
x_k(\alpha)=\frac{x_k+\alpha d_k}{\|x_k+\alpha d_k\|} \quad \mbox {and}\quad d_k=-B_k^{-1}(I-x_kx_k^T)\nabla\phi(x_k).
$$
\STATE Find $\alpha_k\in(0,1]$ by a backtracking process satisfying
\begin{equation}\label{armijo}
\phi(x_k(\alpha_k))\le\phi(x_k)-\sigma \alpha_k[F(x_k)]^TB_k^{-1}F(x_k).
\end{equation}
\STATE Let $x_{k+1}=x_k(\alpha_k)$ and $\lambda_{k+1}={\cal A}x_{k+1}^m$. Let $k:=k+1$.
\ENDWHILE
\end{algorithmic}
\end{algorithm}

Similar to the convergence of a descent method for solving unconstrained optimization problem, we can show the global convergence for Algorithm \ref{alg1}. Since the proof is standard, we only present the conclusion and omit the proof.
\begin{theorem}\label{converg}
Suppose that there are positive constants $0<m\le M$ such that the matrix sequence $\{B_k\}$ in Algorithm \ref{alg1} satisfies
$$
m\|d\|^2\le d^TB_kd\le M\|d\|^2,\quad \forall d\in\mathbb{R}^n.
$$
Then, every accumulation point $\bar{x}$ of the sequence $\{x_k\}$ generated by Algorithm \ref{alg1}
is a $Z$-eigenvector of ${\cal A}$ corresponding to the eigenvalue $\bar{\lambda}={\cal A}\bar{x}^m$.
\end{theorem}

\subsection{A local Newton method and its globalization}\label{opt-newton}

In this subsection, we consider using Newton's method to find a $Z$-eigenpair of a symmetric tensor $\cal A$. We first propose a feasible local Newton method and then, globalize it by using some line search technique.

Recall that the $Z$-eigenvalue problem is equivalent to the constrained nonlinear equations
\begin{eqnarray}\label{equ-equa}
F(x)={\cal A}x^{m-1}-{\cal A}x^m\cdot x=(I-xx^T)\nabla \phi (x)=0,\quad x\in \mathbb B.
\end{eqnarray}
It is indeed a system of nonlinear equations with $n+1$ equations but $n$ variables. In general, such
a system may have no solutions. However, we notice that the function $F(x)$ defined by (\ref{equ-equa})
satisfies $[F(x)]^Tx=0$ for any $x\in \mathbb B$, i.e.,
\[
F(x)\in x^\perp \stackrel\triangle {=} \{d\in \mathbb R^n\;|\; x^Td=0\},\quad\forall x\in \mathbb B.
\]
If we let $U_x\in \mathbb R^{n\times (n-1)}$ be a matrix whose columns form an orthogonal basis of the
subspace $x^\perp$, then, the constrained equations (\ref{equ-equa}) is equivalent to the following system of nonlinear equations:
\begin{equation}\label{temp:equ}
U_x F(x)=0\quad\mbox{and}\quad x^Tx-1=0.
\end{equation}

With the above preparation, we will develop a feasible Newton method with the subproblem being
the following system of linear equations in $d_k$:
\begin{equation}\label{sub:local-N}
U_{x_k}^T  F'(x_k)d_k+U^T_{x_k}F(x_k)=0\quad \mbox{and}\quad x_k^Td_k=0,
\end{equation}
which is obtained by using the approximate relationships
\[
F(x_k)\approx F(x_k)+F'(x_k)d_k\quad \mbox{and}\quad x_k^Tx_k-1\approx (x_k^Tx_k-1)+ x_k^Td_k
\]
and the fact that $x_k^Tx_k=1$, where the Jacobian of $F$ at $x$ is give by
$F'(x)=(m-1){\cal A}x^{m-2}-{\cal A}x^m\cdot I - m x ( {\cal A}x^{m-1})^T$. Then, the Newton direction $d_k$ can be obtained by solving the following lower dimensional system of linear equations:
\begin{equation}\label{sub:local-N-a}
d_k=U_{x_k}u_k,\quad (U_{x_k}^T  F'(x_k)U_{x_k} ) u_k +U^T_{x_k}F(x_k)=0.
\end{equation}
If the matrix $(U_{x_k}^T  F'(x_k)U_{x_k} )$ is nonsingular, then, the Newton direction exists and is unique. In particular, if the second order sufficient condition holds at $x_k$, then, the matrix $(U_{x_k}^T  F'(x_k)U_{x_k} )$ is positive definite, and hence, the Newton direction is the unique solution to (\ref{sub:local-N}) or (\ref{sub:local-N-a}).

Since the Newton step $x_k+d_k$ may not be on the unit sphere $\mathbb B$, we then consider to use its
projection onto $\mathbb B$ to be the next iterate, that is,
\begin{equation}\label{iter-Newton}
x_{k+1}=\frac {x_k+d_k}{\|x_k+d_k\|},\quad k=0,1,\dots
\end{equation}

We call the iterative process (\ref{sub:local-N}) and (\ref{iter-Newton}) a feasible local Newton method
because the sequence of iterates $\{x_k\}$ generated by this method is contained in the feasible set $\mathbb B$.

Now, we are going to analyse the convergence of the proposed local Newton method. To this end, we need the following assumption.
\vspace{2mm}

\noindent
{\bf Assumption (A)} There is a KKT point $\bar x$ of problem (\ref{opt}) satisfying the second order sufficient condition.
In other words, the point $\bar x$ together with some $\bar\lambda\in \mathbb R$ satisfies
\[
\nabla \phi (\bar x) - \bar\lambda \bar x= 0,\quad \bar x\in \mathbb B
\]
and
\[
d^T \Big ( \nabla^2 \phi (\bar x) - \bar\lambda I \Big ) d > 0,\quad \forall d\in \bar x^\perp,\; d\neq 0.
\]

\begin{proposition}\label{prop:second}
The second order sufficient condition stated in Assumption (A) is equivalent to the following conditions:
\begin{description}
\item [\rm{(i)}] The point $\bar x$ is a solution of the constrained equations (\ref{equ-equa}) and $\bar \lambda ={\cal A}\bar x^m$
is the corresponding Lagrangian multiplier.
\item [\rm{(ii)}] The matrix $U_{\bar x}^TF'(\bar x) U_{\bar x}$  is positive definite.
\end{description}
\end{proposition}
\begin{proof}
The Lagrangian function of the problem (\ref{opt}) is denoted by
\[
L(x,\lambda)=\phi (x)- \frac 12 \lambda \Big (x^Tx-1\Big )=\frac 1m {\cal A}x^m-\frac 12 \lambda \Big (x^Tx-1\Big ).
\]
Since $\bar x$ is a KKT point of problem (\ref{opt}), there is a $\bar \lambda $ such that
\[
\nabla _xL(\bar x,\bar\lambda)=\nabla \phi (\bar x) - \bar \lambda x=0 \quad\mbox{and}\quad \bar{x}^T\bar{x}=1,
\]
which implies the condition (i). Besides, the assumption that $\bar x$ satisfies the second order sufficient condition
means that
\[
d^T\nabla _x^2L(\bar x,\bar\lambda)d=d^T \Big ( \nabla ^2\phi (\bar x)-\bar\lambda I\Big )d>0,\quad \forall d: \; \bar x^Td=0, d\neq0.
\]
It is not difficult to see that it is equivalent to condition (ii).
\end{proof}

The last proposition ensures the existence and the uniqueness of the Newton direction $d_k$ when $x_k$
is near $\bar x$. The theorem below establishes the local quadratic convergence of the feasible local Newton method.
\begin{theorem}\label{th:quad-local-N}
Suppose that Assumption (A) holds. If $x_0$ is sufficiently close to a KKT point $\bar x$ of problem (\ref{opt}),
then, the sequence of iterates $\{x_k\}$ generated by the feasible local Newton method converges to $\bar x$ quadratically.
\end{theorem}
\begin{proof}
Denote $P (x)= \left (\begin{array}{c} U_{x}^T F'(x) \\ x^T \end{array} \right )$. By the positive definiteness of
$U_{\bar x}^TF'(\bar x) U_{\bar x}$, it is not difficult to show that the matrix $P(\bar x)$ is nonsingular.
 As a result, the matrix $P(x)$ is uniformly nonsingular in some neighbourhood of $\bar x$.
 It then follows from (\ref{sub:local-N}) that $\|d_k\|=O(\|x_k-\bar x\|)$.

Now, we verify  the relation
\begin{equation}\label{temp:local}
\|x_{k+1}-\bar x\|=O(\|x_k-\bar x\|^2)
\end{equation}
by induction. Suppose that $x_k$ is sufficiently close to $\bar x$ such that the matrix
$U_{x_k}^TF'(x_k)U_{x_k}$ is positive definite. It follows from the first equality of (\ref{sub:local-N}) and $F(\bar{x})=0$ that
\begin{eqnarray*}
U_{x_k}^T F'(x_k)(x_k+d_k-\bar x) &=& -U_{x_k} ^T \Big ( F(x_k)-F(\bar x)-F'(x_k)(x_k-\bar x)\Big ) \\ &=&O(\|x_k-\bar x\|^2).
\end{eqnarray*}
Since $x_k,\bar x\in \mathbb B$, we get from the second equality of (\ref{sub:local-N}) that
\[
x_k ^T(x_k+d_k-\bar x)= x_k^T(x_k-\bar x)=\frac 12 \|x_k-\bar x\|^2 =O(\|x_k-\bar x\|^2).
\]
Therefore, we obtain that
\[
P(x_k) (x_k+d_k-\bar x)= O(\|x_k-\bar x\|^2),
\]
which yields the fact that
\begin{equation} \label{temp:local-a}
\|x_k+d_k-\bar x\|=O(\|x_k-\bar x\|^2)
\end{equation}
as long as $x_k$ is sufficiently close to $\bar x$.

On the other hand, similar to the derivation of (\ref{iter2}), it is easy to derive that
\begin{eqnarray*}
x_{k+1}-\bar x &=&x_k+(I-x_kx_k^T) d_k -\bar x + O(\|d_k\|^2)\\
&=&x_k+d_k-\bar x+O(\|d_k\|^2)=O(\|x_k-\bar x\|^2),
\end{eqnarray*}
which together with (\ref{temp:local-a}) implies (\ref{temp:local}). By the principle of induction,
we claim that $\{x_k\}$ converges to $\bar x$ quadratically.
\end{proof}

The remainder of this subsection is devoted to the globalization of the feasible local Newton method.

\begin{proposition}\label{prop:desc-N}
Suppose that $x_k\in \mathbb B$ is not a KKT point of (\ref{opt}) at which the second order sufficient
condition holds, i.e. $U_{x_k}^TF'(x_k)U_{x_k}$ is positive definite in $\mathbb R^{n-1}$. Then, the Newton direction determined by (\ref{sub:local-N}) is a curve descent direction of $\phi$ at $x_k$.
\end{proposition}
\begin{proof}
We prove the proposition by verifying that the Newton direction $d_k$ satisfies
\[
[\nabla \phi (x_k)]^T (I-x_kx_k^T)d_k<0.
\]
Indeed, it follows from (\ref{sub:local-N-a}) that
\begin{eqnarray*}
[\nabla \phi (x_k)]^T  (I-x_kx_k^T)d_k &=& [F(x_k)]^T d_k = [F(x_k)]^TU_{x_k}p_k\\
    &=& -[F(x_k)]^TU_{x_k} \Big ( U_{x_k} ^TF'(x_k)U_{x_k}\Big )^{-1} U_{x_k}^T F(x_k)\le 0.
\end{eqnarray*}
The last inequality reduces to equality if and only if $U_{x_k}^T  F(x_k)=0$,
which is equivalent to the fact that $(I-x_kx_k)^TF(x_k)=F(x_k)=0$, that is, $x_k$ is a KKT point of (\ref{opt}).
\end{proof}

Proposition \ref{prop:desc-N} provides a way to globalize the feasible local Newton method. Specifically, we replace the direction $d_k$ in Algorithm \ref{alg1} with the Newton direction determined by (\ref{sub:local-N}). Taking into account that when $x_k$ is far away from the KKT point $\bar x$, the second order sufficient condition may not hold, which leads to the non-existence of the Newton direction or non-descent possibility of the Newton direction. In this case, we use the gradient direction $d_k=-F(x_k)=-(I-x_kx_k^T) f(x_k)$ instead of the Newton direction.

The steps of the global Newton method are presented in Algorithm \ref{alg2}.

\begin{algorithm}[!htbp]
\caption{(A global Newton method for problem (\ref{opt})).}\label{alg2}
\begin{algorithmic}[1]
\STATE Given a constant $\sigma\in (0,\frac{1}{2})$ and an initial point $x^0\in \mathbb{B}$.
Let $k:=0$ and $\lambda_0={\cal A}x_0^m$.
\WHILE{$F(x^k)\neq 0$}
\STATE Compute the Newton direction $d_k$ by the local Newton method given by (\ref{sub:local-N}).
\STATE If $d_k$ does not exist or is not a curve descent direction, let $d_k=-(I-x_kx_k^T)f(x_k)$,
or $d_k$ be the solution of the following linear equations system
\begin{equation}\label{dir:LM}
\Big ([F'(x_k)] ^2 +\mu \|F(x_k)\| I\Big  )d _k +F'(x_k)F(x_k)=0,
\end{equation}
where $\mu>0$ is an arbitrary constant.
\STATE Let
$$
x_k(\alpha)=\frac{x_k+\alpha d_k}{\|x_k+\alpha d_k\|}.
$$
Find $\alpha_k\in(0,1]$ satisfying the following inequality by a backtracking process
\begin{equation}\label{armijo2}
\phi(x_k(\alpha_k))\le\phi(x_k)+\sigma\alpha_k [F(x_k)]^Td_k.
\end{equation}
\STATE Let $x_{k+1}:=x_k(\alpha_k)$ and $\lambda_{k+1}={\cal A}x_{k+1}^m$. Let $k:=k+1$.
\ENDWHILE
\end{algorithmic}
\end{algorithm}

The global convergence of Algorithm \ref{alg2} is stated as follows.
\begin{theorem}\label{converg2}
The sequence of iterates $\{x_k\}$ generated by  Algorithm \ref{alg2} is bounded and the sequence
of function evaluations $\{\phi (x_k)\}$ is decreasing. Suppose that there is an accumulation point
$\bar x$ of $\{x_k\}$ at which the second order sufficient condition holds. Then, $\bar x$ is a $Z$-eigenvector of $\cal A$ corresponding to the eigenvalue $\bar \lambda={\cal A}\bar x^m$. Moreover,
every accumulation point $\tilde x$ of $\{x_k\}$ is a $Z$-eigenvector of ${\cal A}$ corresponding to
the eigenvalue $\tilde{\lambda}={\cal A}\tilde x^m$.
\end{theorem}
\begin{proof}
From step 5 of Algorithm \ref{alg2} and Proposition \ref{prop:desc-N}, it is easy to see that $\{x_k\}$ is bounded and $\{\phi (x_k)\}$ is decreasing. Now we show $\bar{x}$ is an
eigenvector of $\cal A$ corresponding to the eigenvalue $\bar \lambda$.

Since $\bar{x}$ is an accumulation point, there exists an integer set $K$ such that $\lim_{k\rightarrow\infty,k\in K}x_k=\bar{x}$. When $k\in K$ is sufficiently large, Algorithm \ref{alg2} reduces to the local Newton method mentioned above. Let $\lim_{k\rightarrow\infty,k\in K}d_k=\bar{d}$, then, from the fact that $\{x_k\}$ is bounded and $\{\phi (x_k)\}$ is decreasing we obtain that
$$
\lim_{k\rightarrow\infty}\alpha_k[F(x_k)]^Td_k=0.
$$
\begin{itemize}
  \item Suppose that $\alpha_k\neq0$, then,
  $$
  \lim_{k\rightarrow\infty,k\in K}[F(x_k)]^Td_k=[F(\bar{x})]^T\bar{d}=0.
  $$
  \item Suppose that $\alpha_k=0$, then,
  \begin{equation}\label{ct}
  \phi(x_k(\rho^{-1}\alpha_k))-\phi(x_k)>\sigma\rho^{-1}\alpha_k[F(x_k)]^Td_k,
  \end{equation}
  where $\rho$ is the constant used in backtracking process. Divide both sides of the inequality (\ref{ct}) by $\rho^{-1}\alpha_k$ and take the limit of $k\in K$, then, by using the mean-value theorem and the fact that $[\nabla\phi(x_k)]^T(I-x_kx_k^T)d_k=[F(x_k)]^Td_k$ we obtain that
  $$
  [F(\bar{x})]^T\bar{d}\geq\sigma[F(\bar{x})]^T\bar{d}.
  $$
  Considering $[F(x_k)]^Td_k\leq0$ and $\sigma\in(0,\frac{1}{2})$, we have that $[F(\bar{x})]^T\bar{d}=0$.
\end{itemize}
Combining the above two aspects together we can obtain $[F(\bar{x})]^T\bar{d}=0$. Then, same as the Proposition \ref{prop:desc-N}, $[F(\bar{x})]^T\bar{d}=0$ is equivalent to the fact that $F(\bar{x})=0$, which means that $\bar{x}$ is a $Z$-eigenvector of ${\cal A}$ with $\bar{\lambda}$ being the corresponding $Z$-eigenvalue.

Besides, since $\{\phi(x_k)\}$ is monotone nonincreasing and bounded, it is easy to see that any accumulation point is the minimizer of the problem (\ref{opt}). Therefore, any accumulation point $\tilde{x}$ of $\{x_k\}$ is a $Z$-eigenvector of ${\cal A}$ with $\tilde{\lambda}$ being the corresponding eigenvalue.
\end{proof}

We conclude this section by establishing the quadratic convergence of Algorithm \ref{alg2}.

\begin{theorem}\label{th:quad-N}
Suppose that the sequence $\{x_k\}$ generated by Algorithm \ref{alg2} converges to $\bar x$ at
which the second order sufficient condition holds. If the positive constant $\sigma $ in the algorithm
is chose to satisfy $\sigma <\frac 12$, then, the convergence rate of $\{x_k\}$ is quadratic. Moreover,
 after finitely many iterations, the steplength $\alpha_k=1$ is always accepted.
\end{theorem}

\begin{proof}
We only need to show that when $k$ is sufficiently large, $\alpha_k=1$ always satisfies the line search condition (\ref{armijo2}).

By the definition of $x_k(\alpha)$, we can easily get
\begin{eqnarray*}
x_k(1)-x_k &=& \frac {x_k+d_k}{\|x_k+d_k\|} - x_k= d_k + \Big (\frac {1 }{\|x_k+d_k\|}-1 \Big ) (x_k + d_k)\\
    &=&  d_k - \frac {\|d_k\|^2 }{\|x_k+d_k\|(1+\|x_k+d_k\|)}(x_k+d_k)\\
    &\stackrel\triangle {=} & d_k-\delta_k\cdot (x_k+d_k),
\end{eqnarray*}
where
\[
\delta_k:=\frac {\|d_k\|^2 }{\|x_k+d_k\|(1+\|x_k+d_k\|)}=\frac 12 \|d_k\|^2+o(\|d_k^2\|).
\]

Since $d_k\in x_k^\perp$, it is easy to get from (\ref{sub:local-N}) that $d_k^T (F'(x_k)d_k+F(x_k))=0$. Consequently,
\[
[\nabla \phi (x_k)]^Td_k= \Big ({\cal A}x_k^{m-1} \Big )^T d_k= -d_k^T[F'(x_k)]^Td_k.
\]
Notice that $[\nabla \phi (x_k)]^Tx_k={\cal A}x_k^m=\lambda_k$, then, we can obtain from the definition of $\delta_k$ that
\begin{eqnarray*}
[\nabla \phi (x_k)]^T (x_k(1)-x_k)&=&[\nabla \phi (x_k)]^T d_k-\delta _k [\nabla \phi (x_k)]^T(x_k+d_k)\\
    &=& - d_k^TF'(x_k)d_k -\frac 12 \|d_k\|^2 [\nabla \phi (x_k)]^Tx_k+o(\|d_k\|^2)\\
    &=&  - d_k^TF'(x_k)d_k -\frac 12 \lambda _k\|d_k\|^2 +o(\|d_k\|^2).
\end{eqnarray*}
In addition, we also have that
\begin{eqnarray*}
(x_k(1)-x_k)^T\nabla ^2 \phi (x_k)(x_k(1)-x_k) &=& d_k^T\nabla ^2 \phi (x_k)d_k + o(\|d_k\|^2)\\
&=&d_k^TF'(x_k)d_k+\lambda_k\|d_k\|^2+o(\|d_k\|^2).
\end{eqnarray*}
Therefore, we derive by Taylor's expansion that for all $k$ sufficiently large,
\begin{eqnarray*}
&&\phi (x_k(1))-\phi (x_k)-\sigma [\nabla \phi (x_k)]^Td_k \\
=&&\phi (x_k(1))- \phi (x_k) +\sigma d_k^TF'(x_k)d_k \\
    =&& [\nabla \phi (x_k)]^T(x_k(1)-x_k)+\frac 12 (x_k(1)-x_k)^T\nabla ^2\phi (x_k) (x_k(1)-x_k)+\\
          && o(\|x_k(1)-x_k\|^2)+\sigma d_k^TF'(x_k)d_k\\
    =&& - d_k^TF'(x_k)d_k -\frac 12 \lambda _k\|d_k\|^2 + \frac 12 d_k^TF'(x_k)d_k+ \frac 12 \lambda_k\|d_k\|^2+\\
    &&\sigma d_k^TF'(x_k)d_k+o(\|d_k\|^2)\\
    =&& -(\frac 12 -\sigma ) d_k^TF'(x_k)d_k+o(\|d_k\|^2).
\end{eqnarray*}
By the use of the second order sufficient condition and the fact that $x^T_kd_k=0$, there is a
positive constant $\eta>0$ such that the inequality $ d_k^TF'(x_k)d_k\ge\eta \|d_k\|^2$ holds for
all $k$ sufficiently large. Consequently, the last equality shows that when $k$ is sufficiently large,
the unit stplength is accepted and the method essentially reduces to the local Newton method.
The quadratic convergence then follows from Theorem \ref{th:quad-local-N}.
\end{proof}

\section{A nonlinear equations based Newton method}\label{EqnNewton}

In this section, we propose a constrained nonlinear equations (\ref{equ-equa}) based Newton method. Define the residual function
\begin{eqnarray}\label{merit}
\theta (x)=\frac 12 \|F(x)\|^2,
\end{eqnarray}
where $F(x)$ is defined by (\ref{equ-equa}). Again, we suppose that $\cal A$ is symmetric. By an elementary deduction, it is easy to get the Jacobian of $F(x)$ as follows
\[
F'(x)=(m-1){\cal A}x^{m-2}-{\cal A}x^m\cdot I - m x \Big ( {\cal A}x^{m-1}\Big )^T.
\]

The relationship between a stationary point of $\theta(x)$ and a $Z$-eigenvector of ${\cal A}$ is given in the following lemma.

\begin{lemma}\label{th:stationary}
A point $x\in {\mathbb B}$ is a stationary point of $\theta (x)$ if and only if it is
a $Z$-eigenvector of $\cal A$ corresponding to the eigenvalue ${\cal A}x^m$.
\end{lemma}
\begin{proof}
Clearly, a $Z$-eigenvector is a stationary of $\theta (x)$. We only need to show that any stationary point in ${\mathbb B}$ is also a $Z$-eigenvector of $\cal A$.

Suppose that $\bar x\in {\mathbb B}$ is a stationary of $\theta (x)$, i.e., $\nabla \theta (\bar x)=0$.
Denote $\bar\lambda ={\cal A}\bar{x}^m$. Notice the fact that $\bar x^TF(\bar x)=0$ and
$\bar x^T{\cal A}\bar x ^{m-2}={\cal A}\bar x^{m-1}$, we get that
\begin{eqnarray*}
0 &=& \bar{x}^T\nabla \theta (\bar x)=\bar x^T [F'(\bar x)]^T F(\bar x)\\
    &=& x^T \Big ((m-1){\cal A}\bar x ^{m-2}- \bar\lambda I-m {\cal A}\bar x^{m-1}\cdot \bar x^T\Big ) F(\bar x)\\
    &=& \Big ((m-1){\cal A}\bar x ^{m-1}- (m+1)\bar\lambda \bar x\Big )^T F(\bar x)\\
    &=& \Big ((m-1){\cal A}\bar x ^{m-1}- (m-1)\bar\lambda \bar x\Big )^T F(\bar x)\\
    &=& (m-1)\|F(\bar x)\|^2.
\end{eqnarray*}
The last equality yields $F(\bar x)=0$. Then, it is equivalent to say that $\bar x$ is a $Z$-eigenvector of $\cal A$ corresponding to the eigenvalue $\bar\lambda$.
\end{proof}

For $x_k\in \mathbb B$, we let $Q_k=I-x_kx_k^T$ and $d_k$ be the Newton direction determined by (\ref{sub:local-N}). Since the columns of $U_{x_k}$ form a basis of $x_k^\bot$, the system (\ref{sub:local-N}) implies
\begin{eqnarray}\label{sub:Newton-eq}
Q_k(F'(x_k)d_k+F(x_k))=0.
\end{eqnarray}
Therefore, we obtain that
\begin{eqnarray}\label{desc:Newton}
[\nabla\theta(x_k)]^Td_k&=&-[F(x_k)]^TF'(x_k)d_k \nonumber \\
&=&\Big ({\cal A}x_k^{m-1}\Big )^TQ_kF'(x_k)d_k \nonumber \\
&=&-\Big ({\cal A}x_k^{m-1}\Big )^T Q_k F(x_k) \nonumber \\
&=&-\|F(x_k)\|^2\le 0.
\end{eqnarray}
If $x_k$ is not an eigenvector, i.e., $F(x_k)\neq 0$, then, it holds that $[\nabla\theta(x_k)]^Td_k<0$.

Denote again
\[
x_k(\alpha)=\frac {x_k+\alpha d_k}{\|x_k+\alpha d_k\|}.
\]
Similar to the derivation of (\ref{mean}), it is not difficult to get that
\[
\theta (x_k(\alpha))= \theta (x_k) + \alpha \nabla \theta (x_k) (I-x_kx_k^T)d_k +O(\|\alpha d_k\|^2).
\]
The inequality (\ref{desc:Newton}) ensures that the Newton direction $d_k$ provides a curve descent direction for $\theta (x)$ at $x_k$ as long as $x_k$ is not an eigenvector of $\cal A$. In other words, we have the following proposition.

\begin{proposition}
Suppose that $x_k\in \mathbb B$ is not a $Z$-eigenvector of $\cal A$. The Newton direction determined by (\ref{sub:local-N}), if exists, is a curve descent direction of $\theta$ at $x_k$.
\end{proposition}

Based on the last proposition, we propose a nonlinear equations based Newton method for finding
a $Z$-eigenpair of a symmetric tensor. The steps of the method are given in Algorithm \ref{alg3}.

\begin{algorithm}[!htbp]
\caption{(A nonlinear equations based Newton method for problem (\ref{merit})).}\label{alg3}
\begin{algorithmic}[1]
\STATE Given a constant $\sigma\in (0,1)$ and an initial point $x^0\in \mathbb{B}$. Let $k:=0$ and $\lambda _0={\cal A}x_0^m$.
\WHILE{$F(x^k)\neq 0$}
\STATE Solve the linear equations system (\ref{sub:local-N}) to get $d_k$. If the system
(\ref{sub:local-N}) is not solvable, then, let $d_k=-\nabla\theta(x_k)$ or the solution of (\ref{dir:LM}). Let
$$
x_k(\alpha)=\frac{x_k+\alpha d_k}{\|x_k+\alpha d_k\|}.
$$
\STATE Find $\alpha_k\in(0,1]$ by a backtracking process satisfying
\begin{equation}\label{armijo3}
\theta (x_k(\alpha_k))\le \theta (x_k) +\sigma \alpha _k[\nabla\theta(x_k)]^Td_k.
\end{equation}
\STATE Let $x_{k+1}=x_k(\alpha_k)$ and $\lambda_{k+1}={\cal A}x_{k+1}^m$. Let $k:=k+1$.
\ENDWHILE
\end{algorithmic}
\end{algorithm}

Similar to the proof of Theorems \ref{converg2} and \ref{th:quad-N}, we can establish the global and quadratic convergence of Algorithm \ref{alg3}.

\begin{theorem}\label{converg3}
Let $\{x_k\}$ be generated by Algorithm \ref{alg3}. If there is an accumulation point $\bar{x}$ of $\{x_k\}$ at which the second order sufficient condition holds, then, $\bar{x}$ is a $Z$-eigenvector of ${\cal A}$ with ${\cal A}\bar{x}^m$ being the corresponding $Z$-eigenvalue. If we further suppose that the whole sequence $\{x_k\}$ converges to $\bar x$, then, the convergence rate of $\{x_k\}$ is quadratic. Moreover, $\alpha_k=1$ is accepted for all $k$ sufficiently large.
\end{theorem}

\section{Numerical results}\label{num}

In this section, we are going to test the performance of the proposed Newton methods, namely Algorithms \ref{alg2} and \ref{alg3}, for finding one $Z$-eigenpair of a symmetric tensor $\cal A$. All numerical experiments are done in MATLAB R2016a on a personal laptop with Intel(R) Core(TM) CPU i5-8250U @1.80GHz and 8 GB memory running Microsoft Windows 10. Throughout, we employ the tensor toolbox \cite{TensorT} to proceed tensor computation.

The test tensors come from \cite{CZ13,NW14,ZNG20}. Details are given below.

\vspace{2mm}

\noindent{\bf Problem 1.} (\cite{CZ13}) Tensor ${\cal A}\in\mathbb{S}^{[4,2]}$ is a symmetric and irreducible
nonnegative tensor with the elements given by
\[
a_{1111}=a_{2222}=\frac{4}{\sqrt{3}}, a_{1112}=a_{1121}=a_{1211}=a_{2111}=1,
\]
\[
a_{1222}=a_{2122}=a_{2212}=a_{2221}=1
\]
 and $a_{i_1i_2i_3i_4}=0$ elsewhere.

\noindent{\bf Problem 2.} (\cite{ZNG20}) Tensor ${\cal A}\in\mathbb{S}^{[4,3]}$ is nonnegative symmetric which is defined by
\[
a_{1111}=0.2883,\quad a_{1112}=0.0031,\quad a_{1113}=0.1973,\quad a_{1122}=0.2485,
\]
\[
a_{1123}=0.2939,\quad a_{1133}=0.3847,\quad a_{1222}=0.2972,\quad  a_{1223}=0.1862,
\]
\[
a_{1233}=0.0919,\quad a_{1333}=0.3619,\quad a_{2222}=0.1241,\quad a_{2223}=0.3420,
\]
\[
a_{2233}=0.2127,\quad a_{2333}=0.2727,\quad a_{3333}=0.3054.
\]

\noindent{\bf Problem 3.} (\cite{NW14}) The elements of the symmetric tensor ${\cal A}\in\mathbb{S}^{[3,n]}$ are given by
$$
a_{i_ii_2i_3}=\frac{(-1)^{i_1}}{i_1}+\frac{(-1)^{i_2}}{i_2}+\frac{(-1)^{i_3}}{i_3},\quad \forall i_1,i_2,i_3\in[n].
$$

\noindent{\bf Problem 4.} (\cite{NW14}) The elements of the symmetric tensor ${\cal A}\in\mathbb{S}^{[4,n]}$ are given by
$$
a_{i_ii_2i_3i_4}=\tan(i_1)+\tan(i_2)+\tan(i_3)+\tan(i_4),\quad \forall i_1,\ldots,i_4\in[n].
$$

\noindent{\bf Problem 5.} (\cite{NW14}) The elements of the symmetric tensor ${\cal A}\in\mathbb{S}^{[5,n]}$ are given by
$$
a_{i_ii_2i_3i_4i_5}=(-1)^{i_1}\log(i_1)+\cdots+(-1)^{i_5}\log(i_5),\quad \forall i_1,\ldots,i_5\in[n].
$$

While doing numerical tests, we set the initial point $x_0$ to be a random nonnegative vector
in the feasible set $\mathbb{B}$. The parameter $\sigma$'s in (\ref{armijo2}) and (\ref{armijo3}) are
set to be $0.01$ and $0.005$, respectively. The backtracking line search is implemented by letting
$\alpha_k=\max\{0.1^i,\; i=0,1,\ldots\}$ and $\alpha_k=\max\{0.073^i,\; i=0,1,\ldots\}$ such that $\alpha_k$
satisfies line search conditions in Algorithms \ref{alg2} and \ref{alg3}. For all test problems and any method,
we set the termination criterion to be
\begin{equation}\label{res}
Res:=\|F(x_k)\|=\|(I-x_kx_k^T){{\cal A}}x^{m-1}_k\|\le10^{-10}.
\end{equation}
We also stop the iterative process if the number of iterations reaches to 300 but the termination condition (\ref{res}) is not satisfied. In this case, we regard the method as failing to address that problem.

We are going to use the total number of iterations and the computation time used to obtain one $Z$-eigenpair to mainly indicate the efficiency of these two algorithms. For each test problem, we test both algorithms with 100 different initial points. Tables \ref{tab1}-\ref{tab4} list the average performance of the two methods. The columns of those tables take the following meaning:
\begin{center}
\begin{tabular}{ll}\\ \hline
``iter": &  the average number of iterations;\\
``iter-n"; & the average number of non-Newton steps used;\\
``in-iter": &  the average number of line search;\\
``Cpu": &  the average computing time in seconds;\\
``Res": &  the average last residual given by (\ref{res});\\
``suc\%": & the rate of successful terminations.\\ \hline
\end{tabular}\vspace{3mm}
\end{center}

We first test the performance of Algorithms \ref{alg2} and \ref{alg3} on Problems 1 and 2. As the largest
$Z$-eigenvalues of these two tensors are known, which are $3.1754$ and $2.0690$, respectively, we investigate if our methods can find the largest $Z$-eigenvalues of those two problems, although the purposes of these two methods are finding one $Z$-eigenpair of ${\cal A}$. Table \ref{tab1} lists the performance of the two methods, where the column ``Occ.\%" represents frequency of occurrence of the largest $Z$-eigenvalue.

\begin{table}[!htbp]
\caption{Numerical results for Problems 1 and 2}\label{tab1}
{\scriptsize
\def\temptablewidth{1\textwidth}
\begin{tabular*}{\temptablewidth}{@{\extracolsep{\fill}}ccccccccccc}\toprule
Alg. & Prob. & \textsf{$\lambda^*$} & \textsf{iter}& \textsf{iter-n}& \textsf{in-iter} & \textsf{Cpu} & \textsf{Res}& \textsf{suc\%}& \textsf{Occ.\%} \\ \midrule
Algorithm \ref{alg2} & Prob. 1 &  3.1754 & 6.58 & 1.01 & 0.66 & 0.0066 & 3.69e-12 & 97  & 100 \\
                     & Prob. 2 &  2.0690 & 4.53 & 0.36 & 0.38 & 0.0044 & 7.98e-12 & 97  & 100 \\\hline
Algorithm \ref{alg3} & Prob. 1 &  3.1754 & 4.88 & 0    & 0.26 & 0.0054 & 2.54e-12 & 100 & 74 \\
                     & Prob. 2 &  2.0690 & 5.13 & 0    & 0.17 & 0.0052 & 7.82e-12 & 100 & 64 \\
 \bottomrule
\end{tabular*}}
\end{table}

The results in Table \ref{tab1} show that, starting from an arbitrary initial point in the feasible set $\mathbb{B}$, both methods can find a Z-eigenpair of the tensor  with a high successful rate. It is interesting to notice that Algorithm \ref{alg2} can always find the largest $Z$-eigenvalue of ${\cal A}$ if it terminates regularly. In many cases, Algorithm \ref{alg3} can also get the largest $Z$-eigenvalue of the problem.

Next, we test Algorithms \ref{alg2} and \ref{alg3} on Problems 3-5 with different dimension $n$. For each problem with different dimension $n$, we tested both methods  with 100 different initial points $x_0$ which are random nonpositive vectors in the feasible set $\mathbb{B}$. Tables \ref{tab2}, \ref{tab3} and \ref{tab4} show the average performance of the proposed methods. In those tables, we do not list ``suc\%" for Algorithm \ref{alg3} because the rate of successful terminations of Algorithm \ref{alg3} is always 100\%.
\begin{center}
\begin{table}[!htbp]
\caption{Numerical results for Problem 3}\label{tab2}
{\scriptsize
\def\temptablewidth{1\textwidth}
\begin{tabular*}{\temptablewidth}{@{\extracolsep{\fill}}c|cc|ccc}\toprule
& Algorithm \ref{alg2}&& Algorithm \ref{alg3} \\
$n$ & \textsf{iter} / \textsf{iter-n}/ \textsf{in-iter}/ \textsf{Cpu} / \textsf{Res}/ \textsf{suc\%} && \textsf{iter} / \textsf{iter-n}/ \textsf{in-iter}/ \textsf{Cpu} / \textsf{Res} \\ \midrule
  $10$ &  7.01 / 1.48 / 1.88 / 0.0079 / 3.78e-12 / 91  & & 16.57 / 0.01 / 0.02 / 0.0159 / 3.90e-11 \\
  $20$ &  6.94 / 1.57 / 1.85 / 0.0089 / 5.77e-12 / 95  & & 18.82 / 0    / 0    / 0.0176 / 2.01e-11 \\
  $30$ &  6.82 / 1.43 / 1.41 / 0.0101 / 7.76e-12 / 100 & & 18.74 / 0.01 / 0    / 0.0228 / 1.80e-11 \\
  $40$ &  6.74 / 1.46 / 1.24 / 0.0094 / 6.44e-12 / 93  & & 18.86 / 0.01 / 0    / 0.0246 / 2.25e-11 \\
  $50$ &  6.86 / 1.39 / 1.67 / 0.0107 / 7.79e-12 / 94  & & 18.94 / 0    / 0    / 0.0256 / 3.24e-11 \\
  $60$ &  6.67 / 1.44 / 1.53 / 0.0111 / 9.63e-12 / 90  & & 18.88 / 0    / 0    / 0.0270 / 3.23e-11 \\
  $70$ &  6.72 / 1.38 / 1.73 / 0.0149 / 6.38e-12 / 94  & & 19.09 / 0    / 0    / 0.0325 / 3.88e-11 \\
  $80$ &  6.94 / 1.49 / 2.04 / 0.0156 / 5.44e-12 / 98  & & 31.56 / 28.16 / 79.26 / 0.1663 / 3.82e-11 \\
  \bottomrule
\end{tabular*}}
\end{table}

\begin{table}[!htbp]
\caption{Numerical results for Problem 4}\label{tab3}
{\footnotesize
\def\temptablewidth{1\textwidth}
\begin{tabular*}{\temptablewidth}{@{\extracolsep{\fill}}c|cc|cccccc}\toprule
& Algorithm \ref{alg2}&& Algorithm \ref{alg3} \\
$n$ & \textsf{iter} / \textsf{iter-n}/ \textsf{in-iter}/ \textsf{Cpu} / \textsf{Res}/ \textsf{suc\%} && \textsf{iter} / \textsf{iter-n}/ \textsf{in-iter}/ \textsf{Cpu} / \textsf{Res} \\ \midrule
  $10$ &  6.49 / 1.18 / 3.69 / 0.0087 / 4.86e-12 / 95 & & 6.28 / 1.36 / 0.07 / 0.0087 / 3.08e-12  \\
  $20$ &  6.72 / 1.22 / 6.25 / 0.0147 / 1.50e-11 / 95 & & 6.91 / 1.69 / 0.08 / 0.0111 / 1.46e-11 \\
  $30$ &  6.86 / 1.43 / 7.32 / 0.0224 / 2.33e-11 / 91 & & 7.35 / 1.79 / 0.11 / 0.0147 / 2.33e-11\\
  $40$ &  5.88 / 1.17 / 5.90 / 0.0534 / 3.11e-11 / 82 & & 7.73 / 1.97 / 0.17 / 0.0264 / 2.90e-11  \\
  $50$ &  5.65 / 1.01 / 5.18 / 0.0548 / 5.74e-11 / 94 & & 7.54 / 1.80 / 0.12 / 0.0424 / 5.96e-11 \\
  $60$ &  6.07 / 1.00 / 5.30 / 0.1114 / 6.36e-11 / 57 & & 7.75 / 1.77 / 0.15 / 0.0827 / 6.45e-11 \\
  $70$ &  7.18 / 1.00 / 6.06 / 0.2307 / 7.60e-11 / 51 & & 8.18 / 2.05 / 0.14 / 0.1535 / 7.14e-11 \\
  \bottomrule
\end{tabular*}}
\end{table}

\begin{table}[!htbp]
\caption{Numerical results for Problem 5}\label{tab4}
{\footnotesize
\def\temptablewidth{1\textwidth}
\begin{tabular*}{\temptablewidth}{@{\extracolsep{\fill}}c|cc|ccc}\toprule
& Algorithm \ref{alg2}&& Algorithm \ref{alg3} \\
$n$ & \textsf{iter} / \textsf{iter-n}/ \textsf{in-iter} / \textsf{Cpu} / \textsf{Res}/ \textsf{suc\%} &&\textsf{iter} / \textsf{iter-n}/ \textsf{in-iter}/ \textsf{Cpu} / \textsf{Res} \\ \midrule
  $10$ &  6.89 / 1.24 / 3.58 / 0.0129 / 8.38e-12 / 95 & &  23.96 / 0.06 / 0 / 0.0291 / 5.20e-11\\
  $20$ &  8.12 / 1.63 / 6.97 / 0.0603 / 7.92e-12 / 90 & &  22.08 / 1.00 / 0 / 0.0884 / 5.20e-11 \\
  $30$ &  8.25 / 1.65 / 6.34 / 0.3224 / 1.33e-11 / 100 & & 22.89 / 1.00 / 0 / 0.5200 / 4.14e-11 \\
  $40$ &  8.29 / 2.35 / 10.27 / 1.6514 / 2.36e-11 / 91 & & 23.09 / 1.00 / 0 / 2.0800 / 2.72e-11 \\
  \bottomrule
\end{tabular*}}
\end{table}

\end{center}

The results in Tables \ref{tab2}-\ref{tab4} show the following feature:
\begin{itemize}
  \item Algorithm \ref{alg3} can always get a $Z$-eigenpair of a symmetric tensor while Algorithm \ref{alg2} can find a $Z$-eigenpair for most symmetric tensors.
  \item In the case both Algorithms \ref{alg2} and \ref{alg3} stop at a $Z$-eigenpair of the tensor, Algorithm \ref{alg2} generally needs less iterative numbers as well as computation time. The sizes of the problem seems not to affect the numbers of iterations much.
\end{itemize}

\section{Conclusion}\label{con}

By reformulating the tensor $Z$-eigenpair problem into an optimization problem, we proposed a class of feasible descent methods for finding a $Z$-eigenpair of a symmetric tensor and established their convergence, among which, a feasible Newton method possesses global and quadratic convergence properties.
Based on an equivalent nonlinear equations to the $Z$-eigenvalue problem, we proposed another Newton method and established its global and quadratic convergence. For both Newton methods, the unit steplength is acceptable for all $k$ sufficiently large. As far as we know, there is no other Newton method that can accept the unit steplength. Our numerical results showed the efficiency of the proposed Newton methods in finding one $Z$-eigenpair of a symmetric tensor. However, the proposed Newton methods depend on the symmetry of the involved tensors. It is interesting to develop globally and quadratically convergent Newton's methods for finding a $Z$-eigenpair of a non-symmetric tensor.

\end{document}